\documentclass[11pt]{amsart}
\usepackage{amsmath,amsfonts,amscd,amssymb,amsthm,epsfig,euscript}
\usepackage{mathrsfs}
\usepackage{csquotes}
\usepackage{amsxtra}
\usepackage{verbatim,epsf}
\usepackage{framed}
\newtheorem{theorem}{Theorem}

\newtheorem{lemma}{Lemma}
\newtheorem{corollary}{Corollary}
\theoremstyle{definition}

\theoremstyle{remark}
\newtheorem{remark}{Remark}
\numberwithin{equation}{section}
\newcommand{\field}[1]{\ensuremath{\mathbb{#1}}}
\newcommand{\CC}{\field{C}}

\newcommand{\PP}{\field{P}}

\newcommand{\curly}[1]{\mathscr{#1}}

\newcommand{\cC}{\curly{C}}
\newcommand{\cD}{\curly{D}}

\newcommand{\cF}{\curly{F}}

\newcommand{\cH}{\curly{H}}

\newcommand{\cL}{\curly{L}}
\newcommand{\cM}{\curly{M}}
\newcommand{\cN}{\curly{N}}

\newcommand{\cP}{\curly{P}}
\newcommand{\cQ}{\curly{Q}}

\newcommand{\cU}{\curly{U}}
\newcommand{\cV}{\curly{V}}

\DeclareMathOperator{\tr}{tr} 
 
\DeclareMathOperator{\Aut}{Aut}
 
\DeclareMathOperator{\GL}{GL} \DeclareMathOperator{\End}{End}
\DeclareMathOperator{\Res}{Res} \DeclareMathOperator{\Ad}{Ad}
 \DeclareMathOperator{\ad}{ad}

\usepackage{hyperref}
\hypersetup{colorlinks,citecolor=blue,plainpages=false,hypertexnames=false}
\begin{document}

\title[Optimum parabolic weight chamber examples]{Optimum weight chamber examples of moduli spaces of stable parabolic bundles in genus 0}


\author[Meneses]{Claudio Meneses}
\address{\noindent Centro de Investigaci\'on en Matem\'aticas A.C., Jalisco S/N,
 Valenciana, \indent C.P. 36023, Guanajuato, Guanajuato, M\'exico}
\curraddr{Mathematisches Seminar, Christian-Albrechts Universit\"at zu Kiel, \indent Ludewig-Meyn-Str. 4, 
24118 Kiel, Germany}
\email{meneses@math.uni-kiel.de} 

\thanks{}

\thanks{}


\date{}

\dedicatory{}
\begin{abstract}
We present an explicit construction of the moduli spaces of rank 2 stable parabolic bundles of parabolic degree 0 over the Riemann sphere, corresponding to ``optimum" open weight chambers of parabolic weights in the weight polytope. The complexity of the different moduli space' weight chambers is understood in terms of the complexity of the actions of the corresponding groups of bundle automorphisms on stable parabolic structures. For the given choices of parabolic weights,  $\mathscr{N}$ consists entirely of isomorphism classes of strictly stable parabolic bundles whose underlying Birkhoff-Grothendieck splitting coefficients are constant and minimal, is constructed as a quotient of a set of stable parabolic structures by a group of bundle automorphisms, and is a smooth, compact complex manifold biholomorphic to $\left(\mathbb{C}\mathbb{P}^{1}\right)^{n-3}$ for even degree, and $\mathbb{C}\mathbb{P}^{n-3}$  for odd degree. As an application of the construction of such explicit models, we provide an explicit characterization of the nilpotent cone locus on $T^{*}\mathscr{N}$ for Hitchin's integrable system.\\

\end{abstract}

\maketitle



\section{Introduction}

A rank 2 \emph{(semi)stable parabolic bundle $E_{*}$ of parabolic degree 0} \cite{MS80} on a compact Riemann surface $\Sigma$, is a holomorphic rank 2 vector bundle $E\to \Sigma$, together with a collection of complete flags  $F_{i} = \{E_{i}\supset L_{i} \supset \{0\}\}$ over the fibers $\{E|_{z_{i}}\}$ of a finite set $S = \{z_{1},\dots z_{n}\}\subset \Sigma$, weighted by real numbers $0 \leq \alpha_{i1} < \alpha_{i2} < 1$ such that
\[
\deg\left( E \right) + \sum_{i=1}^{n}\left(\alpha_{i1} + \alpha_{i2}\right) = 0,
\]
and such that for any line subbundle $L \hookrightarrow E$, 
\[
\deg\left(L\right) + \sum_{i=1}^{n}\alpha'_{i} < 0 \qquad \text{(resp. $\leq$)},
\]
where 
\[
\alpha'_{i} = \left\{
\begin{array}{cr}
\alpha_{i2}\quad & \text{if}\quad L |_{z_{i}} = L_{i},\\\\
\alpha_{i1}\quad & \text{otherwise.}
\end{array}\right.
\]
In order to construct a moduli space $\cN^{ss}$ of semistable parabolic bundles of parabolic degree 0 on $\Sigma$, not only the topology of $E$ must be fixed (i.e., its degree and rank), but also a choice of admissible parabolic weights must be made, resulting in the existence of multiple nonequivalent spaces. 

When $\Sigma = \CC\PP^{1}$, a couple of peculiarities are manifested in the moduli problem: 
(i)  \emph{Not all collections of parabolic weights determine a nonempty moduli space}. In  \cite{Bis98}, I. Biswas studied the necessary and sufficient conditions that a choice of parabolic weights must satisfy for the moduli space $\cN^{ss}$ to be nonempty, thus determining what we call the \emph{weight polytope}, parametrizing parabolic weights for which stable parabolic bundles do exist. 
For every admissible degree $-2n < \deg(E) \leq -2$, the corresponding weight polytope contains multiple semistability walls, whose complement consists of a collection of open chambers, with points defining moduli spaces for which $\cN^{ss} = \cN^{s}$ (in such case, the moduli spaces turn out to be smooth, compact $(n-3)$-dimensional complex projective manifolds), and the holomorphic type of $\cN^{s}$ is an invariant of the weight chamber \cite{BH95}. We will only consider parabolic weights in such open chambers of the weight polytope, and we will denote $\cN^{s}$ simply by $\cN$.

Moreover, (ii) since $E \cong \mathcal{O} (m_{1}) \oplus \mathcal{O}(m_{2})$ \cite{Groth57}, the group $\Aut\left(E\right)$ is \emph{always a nontrivial Lie group}. The group of parabolic automorphisms of a stable parabolic bundle $E_{*}$ consists of all nonzero multiples of the identity, which act trivially on any parabolic structure. Thus, 
the residual group 
\[
\Aut\left(E\right)/\text{Par}\Aut(E) = P\left(\Aut\left(E\right)\right),
\]
is simply the projectivization of the group of bundle automorphisms of the underlying vector bundle $E$. The parabolic stability of $E_{*}$ necessarily implies that $H^{0}\left(\CC\PP^{1},E\right) = 0$, hence not only $\deg(E)  < 0$, but also 
\[
m_{1},m_{2}<0
\]

Even in the simplest rank 2 case under consideration, a plethora of weight chambers occurs, and suitable combinatorial tools must be developed in order to understand the resulting intricate complex manifold ``taxonomy".\footnote{There exist several examples in the literature (eg. \cite{Bauer91,Muk05,MY16}) where the relation of the variations of parabolic weights in genus 0 with the minimal model program and Hilbert's 14th problem are considered. We do not pursue such approach here. Casagrande \cite{Ca15} considers another explicit geometric model for a special choice of parabolic weights.} The coarsest holomorphic invariant that one can associate to an open weight chamber is the Harder-Narasimhan stratification of $\cN$, corresponding to the stratification by different splitting types of given degree admitting stable parabolic structures for such weights. When there is a single stratum, $\cN$ can be moreover realized as a quotient of a stable locus in $\left(\CC\PP^{1}\right)^{n}$ under the action of $\Aut(E)$. This article is devoted to solving the following problem:

\vspace{2mm}

\begin{displayquote}
\emph{To find optimal conditions that the parabolic weights must satisfy, in order to make such a quotient as simple as possible}.
\end{displayquote}

\vspace{2mm}

A general approach to study variations of parabolic weights and the consequential wall-crossing phenomena could be entirely formulated in terms of the simultaneous actions and orbit spaces of the groups of automorphisms for different admissible splitting types on a common space of $n$-tuples of flags.  The results of the general construction schemes and classification attempts will appear in \cite{MenSpi17}. This work represents a necessary step towards such a construction.

The present work concludes with an application of the constructed moduli space models. Namely, we provide an explicit set of algebraic equations (equations \eqref{eq:nilpotent cone}) that determine the nilpotent cone locus on $T^{*}\cN$ for each of the models of $\cN$ and the parabolic bundle generalization of the celebrated integrable system of Hitchin.



\section{Statement of results}

A Birkhoff-Grothendieck splitting type $\mathcal{O} (m_{1}) \oplus \mathcal{O}(m_{2})$ is called \emph{evenly-split} if 
\[
|m_{1}-m_{2}|\leq 1
\]
For every fixed degree, there is exactly one splitting type that is evenly-split, up to isomorphism. It is not difficult to prove \cite{Bis98} that the set of isomorphism classes of stable parabolic bundles in $\cN$ whose underlying vector bundle is isomorphic to an evenly-split bundle is nonempty and Zariski open. In general, $\cN$ would stratify according to the Harder-Narasimhan filtration types of the underlying vector bundles of its points. 

Let $k = \left[\sum_{i=1}^{n}\left(\alpha_{i1} + \alpha_{i2}\right)/2\right] = \left[ -\deg\left(E\right)/2\right] > 0$. Depending of the parity of their degree, evenly-split bundles are isomorphic to 
\begin{equation}\label{eq:splitting-type}
E \cong \left\{
\begin{array}{lr}
\mathcal{O}(-k)\oplus\mathcal{O}(-k) & \text{(even degree)}\\\\
\mathcal{O}(-(k+1))\oplus\mathcal{O}(-k) &  \text{(odd degree)}
\end{array}
\right.
\end{equation}
and moreover, any two such isomorphisms differ by postcomposition by an automorphism of the splitting. Let us fix a given isomorphism, once and for all.

\begin{lemma}
If a set of parabolic weights in the weight polytope satisfies the inequality
\begin{equation}\label{eq:ineq-even}
\sum_{i=1}^{n}\alpha_{i1} > \left[ -\deg\left(E\right)/2\right] - 1
\end{equation}
then, for every $\{E_{*}\}\in\cN$, the underlying bundle $E$ is evenly-split.
\end{lemma}

\begin{proof} 
It follows from \eqref{eq:splitting-type} that a splitting type that is not evenly-split would contain a subbundle isomorphic to $\mathcal{O}\left(- k + 1\right)$ (i.e., $\mathcal{O}(k - 1)\otimes E$ would have nowhere zero sections). The definition of parabolic stability and inequality \eqref{eq:ineq-even} implies that such subbundle would destabilize any parabolic structure with such parabolic weights in the given splitting type. 
\end{proof}

For the endomorphism bundle 
$\ad(E) := E^{\vee}\otimes E$, we have that, in both cases,
\[
\dim \End(E) := \dim H^{0}\left(\CC\PP^{1},\ad(E)\right) = 4
\]
Let us choose a basis of $\End(E)$ preserving the splitting of $E$. It readily follows that in terms of such basis, and for each $i =1,\dots, n$,  
\begin{equation}\label{eq:auto-restriction}
\Aut(E)|_{z_{i}} \cong \left\{
\begin{array}{lr}
\mathrm{GL}(2,\CC) & \text{(even degree)}\\\\
\mathrm{B}(2) &  \text{(odd degree)}
\end{array}
\right.
\end{equation} 
where $\mathrm{B}(2)$ is the group of $2\times 2$ invertible lower triangular matrices. 
For each $i =1,\dots, n$, let $\cF_{i} = \PP\left(E|_{z_{i}}\right)$ denote the flag manifold of the fiber $E|_{z_{i}}$. 
It follows that when $\deg(E)$ is even, the action of the group $P\left(\Aut(E)\right)$ on each $\cF_{i}$ by restriction is projective linear. Moreover, the space of subbundles $\mathcal{O}(-k)\hookrightarrow \mathcal{O}(-k)\oplus \mathcal{O}(-k)$ is isomorphic to $\CC\PP^{1}$, and hence provides an identification of all flag manifolds $\cF_{i}$ with $\CC\PP^{1}$, which is independent of the choice of isomorphism \eqref{eq:splitting-type}. Consequently, the actions of $P\left(\Aut(E)\right)$ on each $\cF_{i}$ get identified.

In the case when $\deg(E)$ is odd, there is a unique line subbundle of $E$ isomorphic to $\mathcal{O}(-k)$. Therefore, in such case, the bundle splitting determines a special ``infinity" line singled out, and consequently, each flag manifold $\cF_{i}$ admits a Schubert-Bruhat stratification 
\[
\cF_{i} = \cF_{i}^{0}\sqcup\{\mathcal{O}(-k)|_{z_{i}}\} \cong \CC\sqcup\{\infty\}
\]
The left action of $\Aut\left(E\right)$ on $\cF_{i}$ leaves the line $\mathcal{O}(-k)|_{z_{i}}$ fixed. If we let $u_{i}$ be the complex affine coordinate in $\cF_{i}^{0}$ given by the isomorphism $\cF_{i} \cong \CC\PP^{1}$, the action of any $g\in\Aut(E)$ on $\cF_{i}^{0}$ takes the form 
\begin{equation}\label{eq:odd-degree-action}
g|_{z_{i}}\cdot u_{i} =\left(b_{22}u_{i} + b_{i}\right)/b_{11}\qquad b_{11},b_{22} \in\CC^{*},\quad b_{i}\in\CC
\end{equation}
Consequently, there is an isomorphism
\begin{equation}\label{eq:automorphism-group}
P\left(\Aut\left(E\right)\right) \cong \left\{
\begin{array}{lr}
\mathrm{PSL}(2,\CC) & \text{(even degree)}\\\\
\left(\mathrm{N}(2) \times \mathrm{N}(2)\right)\rtimes \CC^{*} &  \text{(odd degree)}
\end{array}
\right.
\end{equation}
depending on a choice of a pair of points in $\{z_{1},\dots,z_{n}\}$ in the odd degree case (say $z_{n-1}$ and $z_{n}$, without any loss of generality), where $\mathrm{N}(2)\cong \CC$ denotes the complex 1-dimensional group of $2\times 2$ lower unipotent matrices, corresponding to the restriction of the subgroup of unipotent automorphisms of $P(\Aut(E))$ to the flag manifolds $\cF_{n-1}, \cF_{n}$. 


We have arrived at the following conclusion. If for a given fixed admissible degree, inequality \eqref{eq:ineq-even} is imposed, 
every stable parabolic bundle would be isomorphic to a bundle splitting type \eqref{eq:splitting-type}, together with an $n$-tuple of flags in 
\[
\cF_{1}\times\dots\times \cF_{n}
\]
Let us denote by $\cP^{s} \subset \cF_{1}\times\dots\times \cF_{n}$ the subset of $n$-tuples of flags that determine a stable parabolic bundle. Consequently there is an isomorphism 
\begin{equation}\label{eq:quotient}
\cN \cong P\left(\Aut\left(E\right)\right)\setminus \cP^{s}
\end{equation}

Since we understand the action of $P\left(\Aut(E)\right)$ on $\cF_{1}\times\dots\times \cF_{n}$ explicitly, it is natural to present candidates for the stable loci that induce optimal quotient spaces for each degree parity. Namely, let us consider, in the case when $\deg(E) = -2k$ is even, the locus  $\cL \subset \cF_{1}\times\dots\times \cF_{n}$, with respect to $n-3$ points in $\{z_{1},\dots,z_{n}\}$, say $z_{1},\dots,z_{n-3}$, for which the projection
\[
\mathrm{pr}_{1} : \cL \to \cF_{1}\times\dots\times\cF_{n-3}
\] 
is surjective, while the image of the projection
\[
\mathrm{pr}_{2} : \cL\to \cF_{n-2}\times\cF_{n-1}\times\cF_{n}
\]
equals $\cC_{3}$, the configuration space of triples in $\CC\PP^{1}$, under the previously given isomorphisms $\cF_{i}\cong \CC\PP^{1}$.

Moreover, in the case when $\deg(E) = -(2k+1)$ is odd, let us consider the locus $\cL\subset \cF_{1}\times\dots\times \cF_{n}$ 
\[
\cL = \cF^{0}_{1}\times\dots\times\cF^{0}_{n}\setminus \{\Aut(E)\cdot(0,\dots,0)\}
\]
where $\Aut(E)\cdot(0,\dots,0)$ denotes the orbit of $(0,\dots,0)$ under $\Aut(E)$.

\begin{lemma}
The action of the groups $P(\Aut(E))$ on the loci $\cL$ is free and proper. Consequently, $\cL$ is a principal $P(\Aut(E))$-bundle, and moreover, there is an isomorphism
\begin{equation}\label{eq:isomorphism-type}
P(\Aut(E))\setminus\cL \cong \left\{
\begin{array}{lr}
\left(\CC\PP^{1}\right)^{n-3} & \text{(even degree)}\\\\
\CC\PP^{n-3} &  \text{(odd degree)}
\end{array}
\right.
\end{equation}
\end{lemma}

\begin{proof}
The explicit form of the actions makes it clear that they are free and proper. If $\deg(E) = -2k$, then any $n$-tuple $(L_{1},\dots,L_{n})\in\cL$ admits a unique normalization to the special values $L_{n-2} = 1$, $L_{n-1} = 0$, $L_{n} = \infty$ under the left $\textrm{PSL}(2,\CC)$-action. Therefore, 
\[
P\left(\Aut\left(E\right)\right)\setminus\cL \cong \mathrm{pr}_{1}\left(\cL\right) = \cF_{1}\times\dots\times\cF_{n-3}\cong\left(\CC\PP^{1}\right)^{n-3}, 
\]
On the other hand, if $\deg(E) = -(2k+1)$,  the action of $\Aut(E)$ on $\cL$ allows a normalization locus given by 
\[
u_{n-2} = u_{n-1} =0,
\]
whose stabilizer is the residual group of projective diagonal automorphisms, which is isomorphic to $\CC^{*}$. Then, it follows that
\[
\mathrm{N}(2)\times\mathrm{N}(2)\setminus\cL\cong \CC^{n-2}\setminus\{(0,\dots,0)\}
\]
Moreover, in virtue of \eqref{eq:odd-degree-action},  it readily follows that in such case, we have that
\[
P\left(\Aut(E)\right)\setminus\cL \cong \CC\PP^{n-3}.
\]
\end{proof}


\begin{theorem}\label{theo:1}
When $\deg(E) = -2k$, $k = 1,\dots, n - 1$, and the inequality \eqref{eq:ineq-even} is imposed, it follows that $\cP^{s} = \cL$ if the parabolic weights are chosen in the open chamber of the weight polytope determined  by the 
inequalities  \eqref{eq:ineq-even} and 
\begin{equation}\label{eq:normalization}
\alpha_{i 2} + \alpha_{j 2} +\sum_{l\neq i,j} \alpha_{l1} > k,\qquad i \neq j,\quad i,j\in\{n-2,n-1,n\}
\end{equation}
Consequently, $\cN\cong \left(\CC\PP^{1}\right)^{n-3}$ in such case.
\end{theorem}

\begin{theorem}\label{theo:2}
In the case when $\deg(E) = -\left(2k + 1\right)$, $k = 1,\dots, n-2, $ and the inequality \eqref{eq:ineq-even} is imposed, it follows that $\cP^{s} = \cL$ if the parabolic weights are chosen in the open chamber of the weight polytope determined by the inequalities \eqref{eq:ineq-even} and
\begin{equation}\label{eq:infty}
\alpha_{i2} +\sum_{j \neq i}\alpha_{j1} > k,\qquad i=1,\dots,n.
\end{equation}
\begin{equation}\label{eq:large-Bruhat}
\sum_{i=1}^{n}\alpha_{i1} < k
\end{equation}
Consequently, $\cN \cong \CC\PP^{n-3}$ in such case.
\end{theorem}

\section{Proof of theorems \ref{theo:1} and \ref{theo:2}}\label{sec:proofs}

\begin{proof}[\textbf{Proof of theorem 1}]
The weight inequality  \eqref{eq:ineq-even} not only guarantees that all underlying vector bundles admitting stable parabolic structures with such weights would be evenly-split, but also that the only subbundles that could destabilize a given parabolic structure are isomorphic to $\mathcal{O}(-k)$. Such subbundles are parametrized by the space 
\[
\PP\left(H^{0}\left(\CC\PP^{1},\mathcal{O}\oplus\mathcal{O}\right)\right) \cong \CC\PP^{1}
\]
(in homogeneous flag manifold coordinates, such line subbundles correspond to $[b_{0}:b_{1}]$, for $(b_{0},b_{1})\neq(0,0)$).
The inequalities \eqref{eq:normalization} are equivalent to the inequalities
\begin{equation}\label{eq:compactification}
\sum_{l\neq i,j} \alpha_{l2} + \alpha_{i 1} + \alpha_{j 1} < k, \qquad i \neq j,\quad i,j\in\{n-2,n-1,n\}
\end{equation}
which imply that, in fact, all parabolic structures with arbitrary common components in $\cF_{1},\dots,\cF_{n-3}$ would 
be stable, provided that 
\[
L_{n - 2}\neq L_{n - 1},\qquad L_{n - 2}\neq L_{n},\qquad L_{n - 1}\neq L_{n}
\]
under the isomorphisms $\cF_{n-2}\cong\cF_{n-1}\cong\cF_{n}\cong \CC\PP^{1}$. Hence we 
have the inclusion
\[
\cF_{1}\times\dots\times\cF_{n-3}\times \cC_{3} \subset\cP^{s},
\]
where 
$\cC_{3}$ denotes the configuration space of triples in $\CC\PP^{1}$, thought of as a subspace of $\cF_{n-2}\times\cF_{n-1}\times \cF_{n}$.
Moreover, the inequalities \eqref{eq:normalization} were tailored to ensure that, indeed,
\[
\cF_{1}\times\dots\times\cF_{n-3}\times \cC_{3} = \cP^{s}
\]
in other words, we have that $\cP^{s} = \cL$.

To conclude, it remains to show that the open chamber determined by the inequalities \eqref{eq:ineq-even} and 
\eqref{eq:normalization} is nonempty. 
Let $\epsilon, \delta$ be two real numbers, constrained to satisfy the a priori inequalities
\[
0 < \epsilon, \delta < \text{min}\left\{1,\frac{n}{k} - 1\right\}
\]
Consider the real numbers 
\[
\alpha_{i1} = \left\{
\begin{array}{cc}
\displaystyle\frac{(1 - \epsilon)k}{n}, & i=1,\dots, n-3 \\\\
 \displaystyle\frac{(1 - \delta) k}{n} & i = n-2, n-1, n
\end{array}
\right.
\]
and 
\[
\alpha_{i2} =  \left\{
\begin{array}{cc}
\displaystyle\frac{(1 + \epsilon)k}{n}, & i=1,\dots, n-3 \\\\
 \displaystyle\frac{(1 + \delta) k}{n} & i = n-2, n-1, n
\end{array}
\right.
\]
which are easily seen to determine a collection of honest parabolic weights for a degree $-2k$ vector bundle.  It is readily seen that the inequalities \eqref{eq:normalization} (hence, also \eqref{eq:compactification}) would be satisfied provided that the real numbers $\epsilon$ and $\delta$ are constrained to satisfy
\begin{equation}\label{eq:epsilon-delta}
(n - 3)\epsilon < \delta
\end{equation}
The remaining inequality \eqref{eq:ineq-even} will also follow from \eqref{eq:epsilon-delta} if we also require that
\[
\delta < \frac{n}{4k}.
\]
\end{proof}

\begin{proof}[\textbf{Proof of theorem 2}] 
As in the proof of theorem \ref{theo:1}, the weight inequality  \eqref{eq:ineq-even} not only guarantees that all underlying vector bundles admitting stable parabolic structures with such weights would be evenly-split, but also that the only subbundles that could destabilize a given parabolic structure are either isomorphic to $\mathcal{O}(-k)$ or $\mathcal{O}(-(k + 1))$. Thus, we will consider each case individually. 

We know that there is essentially one degree $-k$ subbundle $\mathcal{O}(-k)\hookrightarrow E$, and the inequalities \eqref{eq:infty} ensure that any parabolic structure for which $L_{i} =  \mathcal{O}(-k)|_{z_{i}}$, for some $i = 1,\dots,n$, would be unstable. Hence, any stable parabolic structure necessarily lies in the open affine subspace $\cF^{0}_{1}\times\dots\times\cF^{0}_{n}$. Moreover, no parabolic structure in $\cF^{0}_{1}\times\dots\times\cF^{0}_{n}$ could be destabilized by the subbundle $\mathcal{O}(-k)$.

Now, using the group of bundle automorphisms $\Aut(E)$, we can normalize any element in $\cF^{0}_{1}\times\dots\times\cF^{0}_{n}$ to $u_{n - 1} = u_{n} = 0$. Inequality \eqref{eq:large-Bruhat} implies that a degree $-(k+1)$ subbundle $\iota : \mathcal{O}(-(k + 1))\hookrightarrow E$ would destabilize a parabolic structure if and only if $\iota\left(\mathcal{O}(-(k + 1))\right)|_{z_{i}} = L_{i}$ for all $i = 1,\dots, n$. The normalization of the parabolic structure implies there is exactly one degree $-(k+1)$  subbundle with that property, namely, the first direct summand $\mathcal{O}(-(k + 1))$. Consequently, the only unstable parabolic structure in $\cF^{0}_{1}\times\dots\times\cF^{0}_{n-2}\times\{0\}\times\{0\}$ corresponds to $(0,\dots,0)$, which is stabilized by the residual $\CC^{*}$-action of $P\left(\Aut(E)\right)$. Consequently, in the present case, it also follows that $\cP^{s} = \cL$.

It remains to show that the open chamber determined by the inequalities \eqref{eq:ineq-even} and \eqref{eq:infty}--\eqref{eq:large-Bruhat} is nonempty. To see this, let us consider the real numbers 
\[
\alpha_{i1} = \frac{k(1-\epsilon)}{n},\qquad \alpha_{i2} = \frac{k(1 + \epsilon) + 1}{n},\qquad i=1,\dots,n
\]
The inequality \eqref{eq:large-Bruhat} implies that  $k + 1 < n$, i.e., $-2n+1 < \deg(E)$. Then, it is straightforward to verify that such numbers will define honest parabolic weights, and will satisfy all of the inequalities \eqref{eq:ineq-even} and \eqref{eq:infty}--\eqref{eq:large-Bruhat} if the following constraint is imposed
\[
0 < \epsilon < \frac{1}{k(n-2)}.
\]
\end{proof}

\begin{remark}
We have adopted the terminology and conventions of \cite{MS80} to define parabolic degree and stability, which are the original ones. Under the Mehta-Seshadri correspondence, there is an equivalence between isomorphism classes of stable parabolic bundles of parabolic degree 0 with prescribed parabolic weights, and isomorphism classes of irreducible unitary representations of 
\[
\pi_{1}\left(\CC\PP^{1}\setminus S, z_{0}\right)\cong \langle\gamma_{1},\dots,\gamma_{n}\,;\, \gamma_{1}\cdot\dots\cdot\gamma_{n}\rangle, 
\]
with prescribed conjugacy classes 
\[
\rho\left(\gamma_{i}\right) \in \Ad\left(\mathrm{PSU}(2)\right)\begin{pmatrix}
e^{2\pi\sqrt{-1}\alpha_{i1}} & 0\\
0 & e^{2\pi\sqrt{-1}\alpha_{i2}} 
\end{pmatrix}
\]
However, most of the treatments of the moduli problem for parabolic bundles over the Riemann sphere (eg. \cite{Bauer91,Muk05,MY16}) consider the additional restriction 
\begin{equation}\label{eq:SU(2)}
\alpha_{i2} = 1- \alpha_{i1} > 1/2,
\end{equation}
i.e., $\rho(\gamma_{i})\in \mathrm{SU}(2)$ for every $i=1,\dots,n$. In general, the corresponding weight polytopes in each case are intrinsically different. For instance, condition \eqref{eq:SU(2)} constrains the parity of the degree of $E$ to be equal to the parity of the number of flags in the parabolic structure. From the character variety point of view, it is not always possible to reduce a unitary representation to a special unitary one: to construct such a reduction, it is a necessary  and sufficient condition that a square root of $\det\left(\rho\right)$ exists. In particular, under our hypotheses, the construction of the optimum open weight chambers, and consequently the validity of theorems \ref{theo:1} and \ref{theo:2}, works for all $n\geq 4$. When the restriction \eqref{eq:SU(2)} is imposed, the optimum open weight chambers exist only if $n>4$.
\end{remark}

\section{Nilpotent cone loci}\label{sec:nilpotent}

A \emph{parabolic Higgs field} $\Phi$ on a parabolic bundle $E_{*}$ is a meromorphic $\End(E)$-valued differential on $\CC\PP^{1}$, holomorphic on $\CC\PP^{1}\setminus\{z_{1},\dots,z_{n}\}$, and with simple poles at each $z_{i}$ whose residues belong to the Lie algebras 
\[
\mathfrak{n}(F_{i})\subset\mathfrak{gl}\left(E|_{z_{i}}\right) 
\]
of the unipotent radicals for the parabolic subgroups $\mathrm{P}(F_{i})\subset \GL\left(E|_{z_{i}}\right)$ stabilizing each flag $F_{i}$. The space of parabolic Higgs fields on $E_{*}$
is equivalently determined as
\[
H^{0}\left(\CC\PP^{1},\text{Par}\End(E_{*})^{\vee}\otimes K_{\CC\PP^{1}}\right)
\]
where $\text{Par}\End(E_{*})\to \CC\PP^{1}$ is the vector bundle associated to the subsheaf of endomorphisms of $E$ preserving the parabolic structure of $E_{*}$. Consequently, the space of parabolic Higgs fields is dual to the space of infinitesimal deformations of $E_{*}$, and when $E_{*}$ is  stable, parabolic Higgs fields model the holomorphic cotangent space $T^{*}_{\{E_{*}\}}\cN$. 
 
The construction of the general Hitchin integral systems on moduli spaces of stable pairs \cite{Hitchin87a} relies on the construction of conjugation invariants of a Higgs field on a vector bundle (with or without a parabolic structure). In rank 2, there are essentially two associated conjugation invariants, namely $\tr(\Phi)$ and $\tr(\Phi^{2})$. Since in genus 0 it follows that $\tr(\Phi)\in H^{0}\left(\CC\PP^{1},K_{\CC\PP^{1}}\right)$, a parabolic Higgs field on $\CC\PP^{1}$ is necessarily traceless. Letting $D = z_{1} + \dots + z_{n}$, it readily follows that the image of any parabolic Higgs field under the quadratic map $\Phi \mapsto \tr\left(\Phi^{2}\right)$ belongs to the $(n - 3)$-dimensional space 
\[
H^{0}\left(\CC\PP^{1},K^{2}_{\CC\PP^{1}}(D)\right)
\]
of meromorphic quadratic differentials on $\CC\PP^{1}$ with at most simple poles on $D$. 
Following Laumon  \cite{Laumon88}, we say that a parabolic bundle $E_{*}$ is \emph{very stable} if it doesn't support nilpotent Higgs fields, i.e., if $\tr(\Phi^{2}) = 0$ implies that $\Phi = 0$. Since for every choice of parabolic weights, $T^{*}\cN$ is Zariski open in the corresponding moduli space of stable parabolic Higgs pairs $\cM$, a first step towards the description of the \emph{nilpotent cone} on  $\cM$ is to provide a precise description of it on $T^{*}\cN$, as well as its ``trace" on $\cN$. By definition, the nilpotent cone locus $\cD$ on $T^{*}\cN$ is the subspace of isomorphism classes of pairs $\{(E_{*},\Phi)\}$ for which $E_{*}$ is parabolic stable and $\tr(\Phi^{2}) = 0$. 

As an application of theorems \ref{theo:1} and \ref{theo:2}, we will provide a characterization of the nilpotent cone locus $\cD$ for the optimum weight chamber moduli space models of rank 2 stable parabolic bundles on $\CC\PP^{1}$. Such characterization is explicit and only depends on the parity of the degree of the underlying evenly-split bundles.\footnote{A more abstract characterization of the nilpotent cone for moduli of parabolic bundles on $\CC\PP^{1}$ is given by Kato in \cite{Kato02}.} In order to do so, it is necessary to provide a suitable model for the cotangent bundles $T^{*}\cN$. Although there is a straightforward construction of $T^{*}\cN$ following the explicit nature of the models for $\cN$, we will consider another description of $T^{*}\cN$ suited to define Hitchin's integrable system, in terms of ``universal" spaces $\cH_{\textrm{even}}$ and $\cH_{\textrm{odd}}$ containing all rank 2 parabolic Higgs fields for each degree parity. 

The fundamental idea is based on the following trivial observation. Given a 2-dimensional vector space $V$, the nilpotent elements in $\mathfrak{sl}(V)$ form a Zariski closed set $\cV$ determined by the equation $\tr\left(B^{2}\right) = 0$, the so-called \emph{nilpotent cone}. There is a projection $\cV\setminus\{0\} \to \cF(V)\cong \PP(V)$ given in terms of the kernel line of each $B\in \cV\setminus\{0\}$, and the inverse image of a given flag $F \in \cF(V)$ equals $\mathfrak{n}(F)\setminus\{0\}$.  Let us consider the space 
\[
\cV_{1}\times\dots\times\cV_{n}
\]
where for each $i = 1, \dots, n$, $\cV_{i}$ is the nilpotent cone in $\mathfrak{sl}\left(E|_{z_{i}}\right)$. Then there are projections $\cV_{i}\setminus\{0\}\to \cF_{i}$ for each $i =1,\dots, n$, and moreover, for every parabolic bundle $E_{*}$ there is an injective map 
\begin{equation}\label{eq:residue map}
\iota_{E_{*}} : H^{0}\left(\CC\PP^{1},\textrm{Par}\End (E_{*})^{\vee}\otimes K_{\CC\PP^{1}}\right) \hookrightarrow \cV_{1}\times\dots\times\cV_{n}
\end{equation}
determined as $\Phi \mapsto (B_{1},\dots,B_{n})$, where 
\[
B_{i} = \Res_{z = z_{i}}(\Phi)\qquad i = 1,\dots, n,
\]
since every parabolic Higgs field on $E_{*}$ with zero residues is necessarily trivial. 

Let us assume from now on, without any loss of generality, that $z_{n - 2} = 1$, $z_{n - 1} = 0$, and $z_{n} = \infty$. Moreover, let us consider an evenly-split bundle together with a choice of isomorphism \eqref{eq:splitting-type}. Such choice implies an identification of each $\cV_{i}$ with the Zariski closed space of $2\times 2$ complex nilpotent matrices. In order to construct the required models for $T^{*}\cN$, it is necessary to define two auxiliary spaces $\cH_{\textrm{even}}$ and $\cH_{\textrm{odd}}$ of admissible parabolic Higgs fields.

\subsection{Even degree}
Define $\cH_{\textrm{even}}$ to be the space of meromorphic matrix-valued differentials of the form
\begin{equation}\label{eq:diff1}
\Phi(z) = \left(\sum_{i = 1}^{n-1} \frac{B_{i}}{z - z_{i}}\right)\mathrm{d}z\qquad B_{i} \in \cV,\quad i = 1,\dots,n - 1
\end{equation}
and such that 
\[
B_{n} = - \displaystyle\sum_{i = 1}^{n - 1}B_{i} \in \cV. 
\]
Then, there is an obvious inclusion 
\[
\cH_{\textrm{even}}\hookrightarrow \cV_{1}\times\dots\times\cV_{n}. 
\]
More importantly, every parabolic Higgs field on a parabolic bundle $E_{*}$ for which $E\cong \mathcal{O}(-k)^{2}$ gives rise to a meromorphic matrix-valued differential \eqref{eq:diff1} by restriction to the affine trivialization over $\CC_{0}\subset\CC\PP^{1}$. It readily follows that any inclusion \eqref{eq:residue map} factors through an intermediate inclusion
\begin{equation}\label{eq:inclusion}
\iota'_{E_{*}}:  H^{0}\left(\CC\PP^{1},\textrm{Par}\End (E_{*})^{\vee}\otimes K_{\CC\PP^{1}}\right) \hookrightarrow \cH_{\textrm{even}}.
\end{equation}

\subsection{Odd degree} The Zariski open subset  $\cV'\subset\cV$ is defined as the union of $0$ and the set of nonzero elements whose kernel line is different from $\infty$. $\cV'$ is parametrized by the matices of the form
\begin{equation}\label{eq:V-coord}
B = c\begin{pmatrix}
- u & 1\\
- u^{2} & u
\end{pmatrix}
\end{equation}
where $u\in\CC \cong \CC\PP^{1}\setminus\{\infty\}$ is a coordinate for the kernel eigenline with induced flag $F$, and $c\in\CC$ is a coordinate for $\mathfrak{n}(F)$. In analogy to $\cH_{\textrm{even}}$, define $\cH_{\textrm{odd}}$ to be the space of meromorphic matrix-valued differentials of the form
\begin{equation}\label{eq:diff2}
\Phi(z) = \left(\sum_{i = 1}^{n-1} \frac{B_{i}}{z - z_{i}} + 
c_{n}\begin{pmatrix}
0 & 0\\
u^{2}_{n} & 0
\end{pmatrix}
\right)\mathrm{d}z
\end{equation}
such that  
\[
\Res_{z = \infty}\Ad\begin{pmatrix}
 z^{k + 1} & 0\\
0 & z^{k} 
\end{pmatrix}\left(\Phi\right) = c_{n}\begin{pmatrix}
- u_{n} & 1\\
- u^{2}_{n} & u_{n}
\end{pmatrix}
\]
which is the local expression over $\CC_{0}\subset \CC\PP^{1}$ of any parabolic Higgs field for any parabolic structure over $\mathcal{O}(-(k+1))\oplus\mathcal{O}(-k)$ restricted to lie in $\cF_{1}^{0}\times\dots\times\cF^{0}_{n}$.

The proof of the following lemma is straightforward, and incorporates the parabolic Higgs fields into a simple geometric model for the cotangent bundles $T^{*}\cN$. 

\begin{lemma}
Consider a choice of parabolic weights as in theorems \ref{theo:1} and \ref{theo:2}. For each corresponding stable locus $\cP^{s}$, let $\cQ^{s}_{\textrm{\emph{even}}}$ $(\text{resp.}\; \cQ^{s}_{\textrm{\emph{odd}}})$ be the set
\[
\cQ^{s}_{\textrm{\emph{even}}} = \left\{(\mathbf{F},\Phi)\in\cP^{s}\times \cH_{\text{\emph{even}}}\; :\; B_{i}\in \mathfrak{n}(F_{i})\right\}\qquad (\text{resp.}\; \cH_{\textrm{\emph{odd}}})
\]
whose defining property is interpreted as an ``incidence correspondence". Then, it follows that 
\[
T^{*}\cN \cong \left\{
\begin{array}{ccc}
\Aut(E)\setminus \cQ^{s}_{\textrm{\emph{even}}} & \cong & \left(T^{*}\CC\PP^{1}\right)^{n - 3}\\\\

\Aut(E)\setminus \cQ^{s}_{\textrm{\emph{odd}}} & \cong & T^{*}\CC\PP^{n - 3}.
\end{array}
\right.
\]
where the actions of $\Aut(E)$ on any given $\Phi$ are determined by the actions on each residue matrix $B_{i}$.
\end{lemma}

\begin{remark}
It is convenient to prescribe normalizations to determine orbit representatives in $\cQ_{\textrm{even}}$ (resp. $\cQ_{\textrm{odd}}$) under the actions of $\mathrm{PSL}(2,\CC)$ (resp. $\mathrm{N}(2)\times \mathrm{N}(2)$). By fixing the values of the lines 
 $L_{n - 2}, L_{n - 1}$ and $L_{n}$ that define the flags $F_{n - 1}$, $F_{n - 1}$ and $F_{n}$, we can use the residue relation \eqref{eq:diff1} to fully prescribe the values of $B_{n - 1}$, $B_{n -1}$ and $B_{n}$ in terms of $B_{1},\dots,B_{n -3}$, and thus determine a complex subspace in $\cQ_{\textrm{even}}$ which is biholomorhic to $T^{*}\cN\cong \left(T^{*}\CC\PP^{1}\right)^{n - 3}$. In the case $n = 4$ we simply recover a model for $T^{*}\CC\PP^{1}$ as the blow-up of the hypersurface $\cV\subset \mathfrak{sl}(2,\CC)$ at $0$ from the residue data at $z_{1}$. More generally, the splitting of $T^{*}\cN$ as a product of blown-up surfaces is obtained by restriction of data (residues and flags) at each $z_{1},\dots,z_{n-3}$. 
 
In the case of $\cQ_{\textrm{odd}}$, letting $u_{n-1} = u_{n} = 0$ in the residue relation \eqref{eq:diff2} yields the linear equations 
\begin{equation}\label{eq:fiber coord}
c_{n - 1} = -\sum_{i = 1}^{n - 2}c_{i},\qquad  c_{n} = \sum_{i = 1}^{n - 2}(z_{n-1} -z_{i})c_{i},\qquad \sum_{i = 1}^{n - 2}c_{i} u_{i} = 0,
\end{equation}
that determine an $(n-3)$-vector subbundle of the trivial $(n - 2)$-bundle over $\CC^{n - 2}\setminus\{(0,\dots,0)\}$ with fiber coordinates $c_{1},\dots, c_{n - 2}$.
The residual $\CC^{*}$-action $u_{i}\mapsto a u_{i}$ implies the reciprocal fiberwise action $c_{i}\mapsto a^{-1}c_{i}$. After taking quotients, the trivial bundle on $\CC^{n - 2}\setminus\{(0,\dots,0)\}$ descends to a vector bundle on $\CC\PP^{n - 3}$ isomorphic to $\mathcal{O}(-1)^{n - 2}$, since over the affine chart $u_{1} \neq 0$, the local holomorphic sections $c_{i} = 1/u_{1}$, $i = 1, \dots, n -2$, define a trivialization of it. Consequently, from the inclusions \eqref{eq:inclusion} we obtain a concrete model for $T^{*}\cN\cong  T^{*}\CC\PP^{n - 3}$ in terms of the dual Euler exact sequence of vector bundles over $\CC\PP^{n - 3}$
\[
0 \to T^{*}\CC\PP^{n - 3} \to \mathcal{O}(-1)^{n - 2}\to \mathcal{O} \to 0.
\]
From now on, we will assume that the previous coordinate normalizations in each degree have been performed, yielding the aforementioned models for $T^{*}\cN$. 
\end{remark}

Let us consider the map $\cV\times\cV \to \CC$ given as $(B,B')\mapsto \tr(BB')$.
A simple computation shows that when $B,B' \neq 0$, $\tr(B B') = 0$ if and only if the corresponding kernel lines coincide, i.e., if their induced flags in $\CC^{2}$ coincide.

\begin{corollary}
For any choice of parabolic weights as in theorems \ref{theo:1} and \ref{theo:2}, the nilpotent cone locus on $T^{*}\cN$ is determined over $\cQ_{\textrm{even}}$ and $\cQ_{\textrm{odd}}$ by the equations 
\begin{equation}\label{eq:nilpotent cone}
\sum _{\substack{j = 1\\ j \neq i}}^{n-1}\frac{\tr(B_{i}B_{j})}{z_{i} - z_{j}} = 0,\qquad i = 1,\dots, n - 3
\end{equation}
\end{corollary}

\begin{proof}
The $(n-3)$-dimensional vector space $H^{0}\left(\CC\PP^{1},K^{2}_{\CC\PP^{1}}(D)\right)$ admits an explicit basis in terms of the meromorphic quadratic differentials $q_{i}$ that are determined over the affine chart $\CC_{0}\subset\CC\PP^{1}$ by means of the rational functions 
\[
P_{i}(z) = \frac{1}{z- z_{i}} + \frac{z_{i} - 1}{z} - \frac{z_{i}}{z - 1},\qquad n = 1,\dots,n-3
\]
Hence, a quadratic differential $q\in H^{0}\left(\CC\PP^{1},K^{2}_{\CC\PP^{1}}(D)\right)$ is zero if and only if and only if its residues at $z_{1},\dots,z_{n-3}$ vanish. The proof follows after a direct computation of such residues for the induced quadratic differential of the normalized matrix-valued meromorphic differentials \eqref{eq:diff1} and \eqref{eq:diff2}.
\end{proof}

\begin{remark}
The introduction of affine coordinates \eqref{eq:V-coord} on the Zariski open subset $\cV'\subset\cV$ and the consequential coordinate description of the pull-back of $T^{*}\cN$ to $\CC^{n - 2}\setminus\{(0,\dots,0)\}$ in the odd degree case by means of equations \eqref{eq:fiber coord} provides an explicit set of equations for the nilpotent cone locus $\cD\subset T^{*}\cN$. It is straightforward to verify that in such case, equations \eqref{eq:nilpotent cone} take the form
\[
c_{i}\left(\sum_{\substack{j = 1\\ j \neq i}}^{n-2}c_{j}\frac{(u_{i} - u_{j})^{2}}{z_{i} - z_{j}} - \sum_{j = 1}^{n - 2}c_{j}\frac{u_{i}^{2}}{z_{i}}\right) = 0,\quad i = 1,\dots, n - 3.
\]
\end{remark}
\begin{remark}
It is straightforward to verify that in the case when $n = 4$ and $\cN \cong \CC\PP^{1}$ independently of the degree parity, the nilpotent cone locus gets identified with the union of $\cN$
and the cotangent spaces over the divisor $D = z_{1} + z_{2} + z_{3} + z_{4}$, as described by Hausel \cite{Hausel98}. In particular, its trace on $\cN$ corresponds to the divisor $D$, providing a sort of ``Torelli's theorem" for moduli spaces of stable parabolic bundles on $\CC\PP^{1}$.
\end{remark}

\begin{remark}\label{remark:higher stratification}
The nilpotent cone locus appears as a special component of a more general stratification of $T^{*}\cN$. Such stratification is of paramount importance and arises in the study of Hitchin's integrable systems.
Namely, the vector spaces $H^{0}\left(\CC\PP^{1},K^{2}_{\CC\PP^{1}}(D)\right)$ can be inductively stratified in terms of the subspaces 
\[
Z_{I} = \left\{q = \sum_{i = 1}^{n - 3}c_{i}q_{i}\in H^{0}\left(\CC\PP^{1},K^{2}_{\CC\PP^{1}}(D)\right) \; \Big| \; c_{i} = 0\quad\text{if}\quad i\in I \right\}
\]
for any $I\subset \{1,\dots n\}$, in the sense that every such $Z_{I}$ is isomorphic to the the vector space 
\[
V_{I^{c}} := H^{0}\left(\CC\PP^{1},K^{2}_{\CC\PP^{1}}\left(D_{I^{c}}\right)\right),\qquad D_{I^{c}} = \sum_{j\in I^{c}}z_{j}
\]
In particular, for each $i = 1,\dots, n-3$, let $Z_{i}$ be the hypersurface corresponding to set $I = \{i\}$. The Zariski open set
\[
\cU = V \setminus \bigcup_{i = 1}^{n - 3} Z_{i}
\]
admits a finer stratification when $n \geq 6$. Namely, any quadratic differential $q\in\cU$ possesses exactly $n - 4$ zeros up to multiplicity. The stratification of $\cU$ then corresponds to  counting the multiplicity of such zeros, which also descends to the Zariski open set $\PP(\cU)\subset \PP(V)$. In particular, the restriction to meromorphic quadratic differentials with only simple zeros lead to a Zariski open subset $\PP(\cU')\subset \PP(\cU)$ which is isomorphic to the $(n - 4)$-dimensional configuration space of $(n - 4)$ points in $\CC\PP^{1}\setminus\{z_{1},\dots,z_{n}\}$. In conclusion: 
\emph{the projective space $\PP(V)$ may be thought of as a compactification of the configuration space $\PP(\cU')$, with an intermediate ambient space given by $\PP(\cU)$.}

The stratification of $\cM$, and in particular of the cotangent bundle $T^{*}\cN$, is then induced from the stratification of $\PP(V)$ under Hitchin's integrable system. In particular, the inverse image of the quasiprojective variety $\PP(\cU')$ is the so-called ``regular locus" $\cM_{0}\subset\cM$.
\end{remark}

\noindent \textbf{Acknowledgments.}
I would like to thank Leon A. Takhtajan for his encouragement, suggestions, and stimulating discussions during the creation of this article, and Leticia Brambila-Paz, for her constructive criticism on the preliminary draft. This work was developed under the partial support of the FORDECyT-CONACyT grant 265667. 

\bibliographystyle{amsalpha}
\bibliography{Minimal}

\end{document}